\documentclass[11pt]{amsart}
\usepackage{amsmath,amsthm,amssymb, enumerate}
            \newcommand{\marginalnote}[1]{}
\usepackage{pdfpages}
\usepackage{hyperref}
\usepackage{graphicx}
\usepackage{epsfig}
\usepackage{color}	
\usepackage{ulem}	
\usepackage{pstricks,pst-tree,pst-node,pst-plot,pstricks-add,pst-slpe} 
\usepackage{tkz-graph}
\usepackage{subfig}
\theoremstyle{plain}

\newtheorem{Thm}{Theorem}[section]

\newtheorem{Lem}[Thm]{Lemma}

\newtheorem{Conj}[Thm]{Conjecture}

\newtheorem*{Thm*}{Theorem}
\theoremstyle{remark}

\theoremstyle{definition}

\newtheorem{Not}[Thm]{Notation}
\newtheorem{Def}[Thm]{Definition}
\newtheorem{Defs}[Thm]{Definitions}

\newcommand{\gen}[1]{\left\langle#1\right\rangle}

\newcommand{\ga}{\Gamma}

\newcommand{\transition}{\delta}

\DeclareMathOperator{\Star}{St}

\usepackage{epigraph}

\usepackage{etoolbox}

\makeatletter
\patchcmd{\epigraph}{\@epitext{#1}}{\itshape\@epitext{#1}}{}{}
\makeatother

\begin{document}
\title{Geodesic growth of right-angled Coxeter groups based on trees}
\author{Laura Ciobanu \& Alexander Kolpakov}
\date{\today}

\maketitle

\begin{abstract}
In this paper we exhibit two infinite families of trees $\{T^1_n\}_{n \geq 17}$ and $\{T^2_n\}_{n \geq 17}$ on $n$ vertices, such that $T^1_n$ and $T^2_n$ are non-isomorphic, co-spectral, with co-spectral complements, and the right-angled Coxeter groups (RACGs) based on $T^1_n$ and $T^2_n$ have the same geodesic growth with respect to the standard generating set. We then show that the spectrum of a tree is not sufficient to determine the geodesic growth of the RACG based on that tree, by providing two infinite families of trees $\{S^1_n\}_{n \geq 11}$ and $\{S^2_n\}_{n \geq 11}$, on $n$ vertices, such that $S^1_n$ and $S^2_n$ are non-isomorphic, co-spectral, with co-spectral complements, and the right-angled Coxeter groups (RACGs) based on $S^1_n$ and $S^2_n$ have distinct geodesic growth. 

Asymptotically, as $n\rightarrow \infty$, each set $T^i_n$, or $S^i_n$, $i=1,2$, has the cardinality of the set of all trees on $n$ vertices. Our proofs are constructive and use two families of trees previously studied by B.~McKay and C.~Godsil.

\bigskip

\maketitle

\end{abstract}

\epigraph{``And then he had a Clever Idea. He would\\ go up very quietly to the Six Pine Trees ...''\\ A.A. Milne, Winnie-the-Pooh}

\section{Introduction}

The \textit{geodesic growth function} of a group $G$ with respect to a finite generating set $S$ counts, for each positive integer $n$, the number of geodesics of length $n$ starting at the identity $1_G$ in the Cayley graph of $G$ with respect to $S$. The \textit{geodesic growth series} is the formal power series that takes the values of the geodesic growth function as its coefficients (see Definition \ref{defgrowth}).

The groups that we consider in this paper, right-angled Coxeter groups, or RACGs, are known to have a regular language of geodesics with respect to the standard generating sets, and therefore rational geodesic growth series (see \cite{JJJ} or \cite[Theorem 4.8.3]{BjornerBrenti} for proofs of these facts). 
Our goal here is to obtain more specific data concerning the geodesic growth of RACGs. Namely, we are interested in extracting information about the geodesic growth series from the defining graph of the group. It is known that non-isomorphic graphs define non-isomorphic RACGs \cite{rad}, and that non-isomorphic RACGs can have equal geodesic growth \cite{AntolinCiobanu}. However, we are interested in knowing how much the similarities or differences between two defining graphs influence the geodesic growth of the corresponding RACGs.   

Two RACGs $G_i = G(\Gamma_i)$, $i=1,2$, with non-isomorphic defining graphs $\Gamma_i$, may have equal standard growth (see Definition \ref{defgrowth} (2)); this can be determined by computing the $f$-polynomials of the graphs $\Gamma_i$ \cite[Proposition 17.4.2]{Davis}. However, their geodesic growth exhibits more subtle properties \cite{AntolinCiobanu}, and in general it is not known which graph theoretic conditions completely determine the geodesic growth of a RACG.
The examples of non-isomorphic graphs $\Gamma_i$ defining RACGs with equal geodesic growth in \cite{AntolinCiobanu} are degree-regular and have cycles. 
In this paper we consider the case when the $\Gamma_i$'s are trees. Although these are some of the simplest classes of graphs, we already encounter a phenomenon that shows a great difference between the behaviour of the standard growth and that of the geodesic growth. The standard growth of each RACG $G_i$ is determined solely by the number of vertices (or edges) in the respective tree $\Gamma_i$, while the geodesic growth can be distinct even for two co-spectral trees (which might have co-spectral complements as well). Recall that two graphs are {\it co-spectral} if the characteristic polynomials of their adjacency matrices are the same.

In this paper we count and compute with the help of an automaton generating the geodesic language in a RACG based on a tree $T$. This automaton reflects some of the path information from $T$. In general, a good deal of combinatorial information about $T$ can be extracted from its spectrum. However, we encounter two rather different behaviours: on one hand we produce (infinitely many) pairs of trees $T_1$ and $T_2$ which are non-isomorphic, co-spectral, with co-spectral complements, and the RACGs $G(T_i)$, $i=1,2$, based on them have the same geodesic growth; on the other hand we obtain (infinitely many) pairs of trees $S_1$ and $S_2$ which are non-isomorphic and co-spectral, with co-spectral complements, whose respective RACGs $G(S_i)$ have distinct geodesic growth. This shows that the spectrum of a tree alone does not determine the geodesic growth of the RACG based on that particular tree.

The following two theorems, proved in Sections \ref{s:racg} and \ref{s:racg-diff-growth}, respectively, are the main results of the paper:

\begin{Thm}\label{thm1}
There exist two families of trees $\mathcal{T}^1_n = \{ T^1_1, T^1_2, \dots \}$ and $\mathcal{T}^2_n = \{ T^2_1, T^2_2, ... \}$ on $n\geq 17$ vertices such that for all $i\geq 1$:
\begin{enumerate}
\item[(1)] $T^1_i$ and $T^2_i$ are not isomorphic, but co-spectral, with co-spectral complements, and
\item[(2)] the RACGs $G(T^1_i)$ and $G(T^2_i)$ have equal geodesic growth series.
\end{enumerate}
\end{Thm}

Theorem \ref{thm1} thus answers positively Question 1 in \cite[Section 8]{AntolinCiobanu}.

\begin{Thm}\label{thm2}
There exist two families of trees $\mathcal{S}^1_n = \{ S^1_1, S^1_2, \dots \}$ and $\mathcal{S}^2_n = \{ S^2_1, S^2_2, ... \}$ on $n\geq 11$ vertices such that for all $i\geq 1$:
\begin{enumerate}
\item[(1)] $S^1_i$ and $S^2_i$ are not isomorphic, but co-spectral, with co-spectral complements, and
\item[(2)] the RACGs $G(S^1_i)$ and $G(S^2_i)$ have distinct geodesic growth series.
\end{enumerate}
\end{Thm}

We note that McKay showed in \cite{McKay} that the trees $T^1_i$ and $T^2_i$ in Theorem~\ref{thm1} are simultaneously co-spectral, have co-spectral complements and co-spectral line graphs. On the contrary, the trees $S^1_i$ and $S^2_i$ from Theorem~\ref{thm2} might have line graphs with distinct spectra. 
Moreover, by a result of McKay \cite{McKay}, in both Theorem~\ref{thm1} and Theorem~\ref{thm2} the cardinality of each family $\mathcal{T}^i_n$ and $\mathcal{S}^i_n$ tends asymptotically to the cardinality of the set $\Upsilon_n$ of all trees on $n$ vertices:
\begin{equation}
\frac{\mathrm{card}\, \mathcal{T}^i_n}{\mathrm{card}\, \Upsilon_n} \rightarrow 1 \mbox{\,\,\, and \,\,\,} \frac{\mathrm{card}\, \mathcal{S}^i_n}{\mathrm{card}\, \Upsilon_n} \rightarrow 1, \mbox{\,\,\, as \,\,\,} n\rightarrow \infty, \ i=1, 2.
\end{equation}
The two theorems above, together with McKay's results, lead us to the following conjecture:

\begin{Conj}
If two trees are simultaneously co-spectral, have co-spectral complements and co-spectral line graphs, then they have the same geodesic growth.
\end{Conj}

Section \ref{s:equal_special} contains details about the computation of several kinds of geodesics, and their numbers, for the trees introduced in Section \ref{s:racg}.

The authors have created a Python code $\mathrm{Monty}$\footnote{this is not the given name of the code, which would be obviously too posh for such a petty thing, but a reference name, which is seemingly good for any Python code.} \cite{Pyth}, which performs the computations needed for the results in this paper by using either SAGE standard routines for symbolic computation of determinants and rational expressions of growth series, or the Berlekamp-Massey algorithm for restoring the rational expression for a growth series from a sufficient number of its coefficients, as a more efficient approach. Our Python code constructs finite-state automata that accept the geodesic languages in RACGs based on triangle-free graphs (in this case, trees), and then proceeds to determine the respective growth series.

\section{Definitions and notation} \label{s:def}

Let $S$ be a finite set and $S^*$ the free monoid on $S$. We identify $S^*$ with the set of words over $S$, that is, finite sequences of elements of $S$. We use 
%$\equiv$ to denote equality of words and 
$| . |$ to denote word length. 

Let $G=\gen{S}$ be a group generated by $S$. For an element $g$ of $G$, denote by $|g|_S$ the word length of $g$ with respect to $S$. Given $w\in S^*$, we denote by $\overline{w}$ the image of $w$ in $G$ under the natural  projection $S^* \rightarrow G$. 

\begin{Def}
A word $w$ over an alphabet $S$ is {\it geodesic} in $G = \langle S \rangle$ if $|w|=|\overline{w}|_S$. The set of geodesics in $G$ with respect to $S$ will be denoted by $Geo(G)$ or $Geo(S)$.
\end{Def}

Let $\ga=\ga(G,S)$ be a simple (no loops, no multiple edges) graph with vertex set $V(\Gamma)=S$ (or simply $V$) and edge set $E(\Gamma)$ (or simply $E$), where $E \subseteq V \times V$.
The RACG based on $\Gamma$ is given by the presentation 
$$\langle s \in S \mid s^2=1 \ \forall s\in S, \ \textrm{and}\  (ss')^2=1, \, \forall \{s, s'\} \in E \rangle.$$ It is easy to see that for any two involutions $s$ and $s'$ the relation $(ss')^2=1$ implies $ss'=s's$. This leads to another possible presentation for RACGs: $\langle s \in S \mid s^2=1, ss'=s's, \, \forall \{s, s'\} \in E \rangle.$

In the present paper we use the same letters for the vertices of $\ga$ and the corresponding generators of the group $G(\ga)$.

The {\it star} of a vertex $v\in V$ in $\Gamma$, denoted by $\Star_{\Gamma}(v)$, or $\Star(v)$ if the ambient graph is clear in the given context, is the set of vertices in $\Gamma$ that are adjacent to $v$. That is, 
 $$\Star_{\Gamma}(v)=\{w \in V \mid \{v, w\} \in E \}.$$

 We now describe a finite deterministic automaton that recognises geodesics in RACGs (see \cite{hu} for definitions of languages and automata). We define such automata in the standard way, as quintuples $(Q,\Sigma, \delta, q_0, F)$, where $Q$ is the finite set of \textit{states}, $\Sigma$ the input alphabet, $\delta$ the \textit{transition function}, $q_0$ the \textit{initial} or \textit{start} state, and $F$ the set of \textit{final} or \textit{accepting} states. The following definition is a simplified version of Proposition 4.1 in \cite{AntolinCiobanu}.

\begin{Def}\label{def:DFA} Let $\Gamma=(V,E)$ be a tree. The deterministic finite state automaton recognising the geodesics in $G(\Gamma)$ is $A=(Q, S, \delta, \{\emptyset\}, F)$, where the set of states is $Q=\{\emptyset\} \cup V \cup E \cup \{ \rho\}$, $\rho$ is the unique ``fail'' state, and $\{\emptyset\}$ is the start state. The input alphabet is $S=V$, the set of accept states $F$ is $\emptyset \cup V \cup E$, and the transition function $\transition:Q \times S \rightarrow Q$ is given by
\begin{enumerate}
\item $\transition(\sigma, s)=(\Star(s) \cap \sigma) \cup \{s\}$, for $s \notin \sigma$; 
\item $\transition(\sigma,s)=\rho$, otherwise.
\end{enumerate} 
\end{Def}
 
\begin{Defs}
Any set $L$ of words over an alphabet $\Sigma$ gives rise to a
 \textit{strict growth function} $f_L:\mathbb{N} \rightarrow \mathbb{N}$, 
defined by  $$f_L(n) := |\{W \in L \mid |W| = n\}|.$$
\end{Defs}

\begin{Defs} \label{defgrowth}
Let $G$ be a group generated by $S$. Then we define the following:
\begin{enumerate}
\item The \textit{geodesic growth function} $f_{Geo(G)}:\mathbb{N}\rightarrow \mathbb{N}$ is given by 
$$f_{Geo(G)}(r):= f_{Geo(G,S)}(r)=|\{w\in S^* \mid |w|=|\overline{w}|_S=r\}|,$$
and the \textit{geodesic growth series} of $G$ equals $$\mathcal{G}_{(G,S)}(t)=\sum_{r=0}^\infty f_{Geo(G)}(r)\, t^r.$$ 

\item The \textit{spherical (standard) growth function} $\sigma_{(G,S)}:\mathbb{N}\rightarrow \mathbb{N}$ is given by 
$$\sigma_{(G,S)}(r)=|\{g\in G \mid |g|_S=r\}|,$$
and the \textit{spherical (standard) growth series} of $G$ equals $$\Sigma_{G}(t):=\Sigma_{(G,S)}(t)=\sum_{r=0}^\infty \sigma_{(G,S)}(r)\, t^r.$$ 
\end{enumerate}
\end{Defs}
 
We recall that the {\it $f$-polynomial} of a graph $\Gamma$ is the generating function for the number of cliques (that is, complete subgraphs) of size $i$ in $\Gamma$: $f(t)=f_0+f_1t+f_2t^2+ \dots $, where $f_i$ is the number of $i$-cliques in $\Gamma$. We consider the empty set to be a clique on zero vertices, and therefore $f_0=1$.
We remark that the spherical (standard) growth function of a RACG is determined by the $f$-polynomial of its defining graph \cite[Proposition 17.4.2]{Davis}. We want to contrast this to the fact that the geodesic growth function is not uniquely determined by the $f$-polynomial, or even by the spectrum of the defining graph, as the following sections show.

%%%%%%%
%SECTION%
%%%%%%%

\section{RACGs with equal geodesic growth}\label{s:racg}

In this section we prove Theorem \ref{thm1} by giving an explicit construction of the families of trees $\mathcal{T}^i_n$, $i=1,2$. 

\begin{Def}
We define the \textit{coalescence} $\tau \cdot \sigma$ of two rooted trees $\tau$ and $\sigma$ to be the tree which results from merging $\tau$ and $\sigma$ at their roots. The tree $\tau \cdot \sigma$ has as vertex set the union of the vertex sets of $\tau$ and $\sigma$, and as root the identification of the roots of $\tau$ and $\sigma$, as shown in Fig.~\ref{fig:trees0}. 
\end{Def}

\begin{figure}[ht]
\begin{center}

\begin{subfloat}{}
\begin{tikzpicture}
\SetVertexNoLabel

\tikzset{VertexStyle/.style = {draw,circle}}
\Vertex{0} 
\SetGraphUnit{2}
\tikzset{VertexStyle/.style = {draw,circle,fill,label = 270:$\tau$}}
\EA(0){1} 
\tikzset{VertexStyle/.style = {draw,circle}}
\SetGraphUnit{1}
\NOEA(1){2}
\NOWE(2){3} \NOEA(2){4}

\Edge(0)(1) \Edge(1)(2)
\Edge(2)(3) \Edge(2)(4)

\SetGraphUnit{3}
\SO(0){A}
\SetGraphUnit{2}
\tikzset{VertexStyle/.style = {draw,circle,fill,label = 90:$\sigma$}}
\EA(A){B}
\tikzset{VertexStyle/.style = {draw,circle}}
\SetGraphUnit{1}
\EA(B){C} \SO(B){D}

\Edge(A)(B) \Edge(B)(C) \Edge(B)(D)

\end{tikzpicture}
\end{subfloat}
~
\begin{subfloat}{}
\begin{tikzpicture}
\SetVertexNoLabel

\tikzset{VertexStyle/.style = {draw,circle}}
\Vertex{0} 
\SetGraphUnit{2}
\tikzset{VertexStyle/.style = {draw,circle,fill,label = 225:$\tau \cdot \sigma$}}
\EA(0){1} 
\tikzset{VertexStyle/.style = {draw,circle}}
\SetGraphUnit{1}
\NOEA(1){2}
\NOWE(2){3} \NOEA(2){4}

\Edge(0)(1) \Edge(1)(2)
\Edge(2)(3) \Edge(2)(4)

\SetGraphUnit{2}
\SO(1){A}
\SetGraphUnit{0.75}
\NOWE(1){B} 
\SetGraphUnit{1}
\EA(1){C}

\Edge(1)(A) \Edge(1)(B) \Edge(1)(C)
\end{tikzpicture}
\end{subfloat}
\end{center}
\caption{Trees $\tau$ and $\sigma$ with marked roots (on the left) and their coalescence $\tau \cdot \sigma$ (on the right)}\label{fig:trees0}
\end{figure}
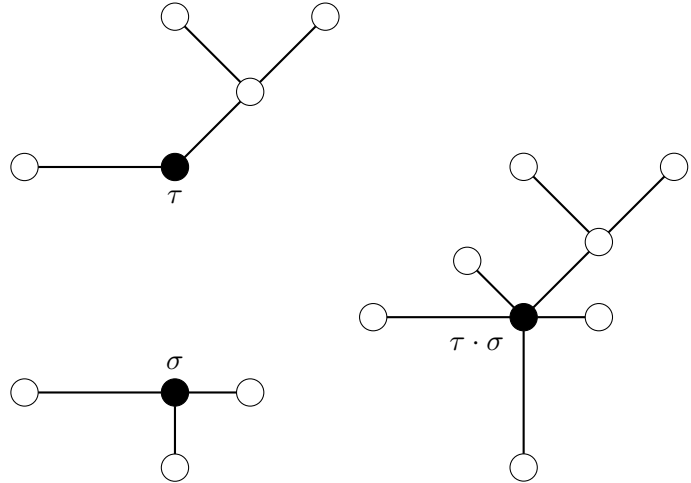

Consider the trees $T_1$ and $T_2$, both rooted at $0$, first described by McKay in \cite{McKay}, as shown in Fig.~\ref{fig:trees1}. In \cite{McKay} the following fact is proved:

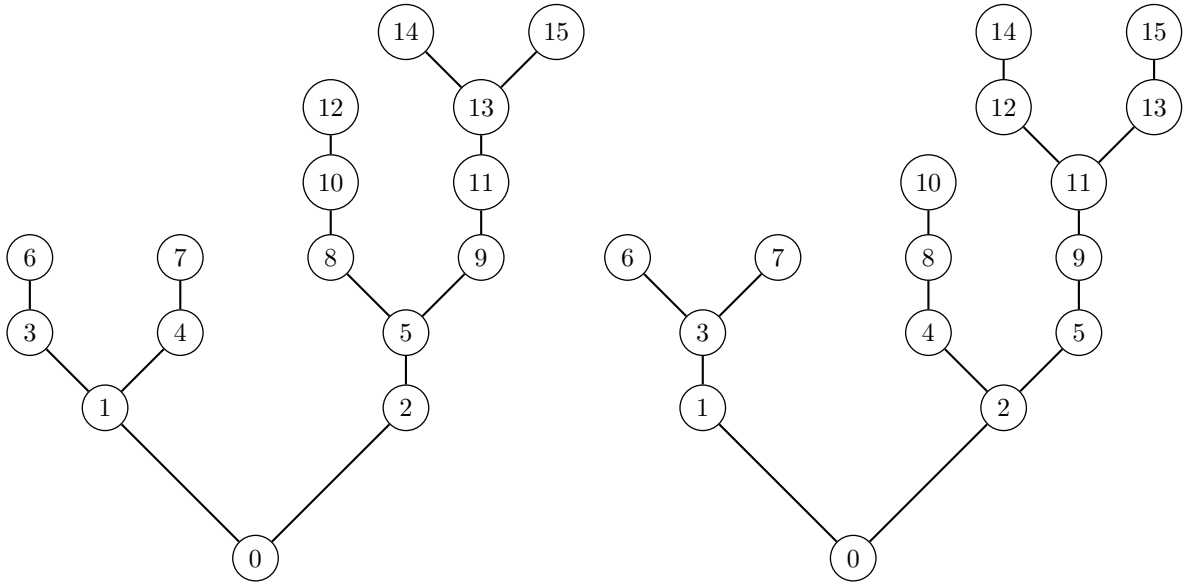
\begin{figure}[ht]%\scalebox{.5}
\begin{center}

\begin{subfloat}{}
\begin{tikzpicture}[every node/.style={scale=0.9}]
\GraphInit[vstyle=Normal]
\SetGraphUnit{2}
\Vertex{0}
\NOWE(0){1} \NOEA(0){2} 
\Edge(0)(1) \Edge(0)(2)
\SetGraphUnit{1}
\NOWE(1){3} \NOEA(1){4} \NO(2){5}
\Edge(1)(3) \Edge(1)(4) \Edge(2)(5)
\NO(3){6} \NO(4){7} \NOWE(5){8} \NOEA(5){9}
\Edge(3)(6) \Edge(4)(7) \Edge(5)(8) \Edge(5)(9)
\NO(8){10} \NO(9){11} \NO(10){12} \NO(11){13}
\Edge(8)(10) \Edge(9)(11) \Edge(10)(12) \Edge(11)(13)
\NOWE(13){14} \NOEA(13){15}
\Edge(13)(14) \Edge(13)(15)
\end{tikzpicture}
\end{subfloat}
~
\begin{subfloat}{}
\begin{tikzpicture}[every node/.style={scale=0.9}]
\GraphInit[vstyle=Normal]
\SetGraphUnit{2}
\Vertex{0}
\NOWE(0){1} \NOEA(0){2} 
\Edge(0)(1) \Edge(0)(2)
\SetGraphUnit{1}
\NO(1){3} \NOWE(2){4} \NOEA(2){5}
\Edge(1)(3) \Edge(2)(4) \Edge(2)(5)
\NOWE(3){6} \NOEA(3){7} 
\Edge(3)(6) \Edge(3)(7) 
\NO(4){8} \NO(5){9} \NO(8){10} \NO(9){11}
\Edge(4)(8) \Edge(5)(9) \Edge(8)(10) \Edge(9)(11)
\NOWE(11){12} \NOEA(11){13}
\Edge(11)(12) \Edge(11)(13)
\NO(12){14} \NO(13){15}
\Edge(12)(14) \Edge(13)(15)
\end{tikzpicture}
\end{subfloat}
\end{center}
\caption{McKay's rooted trees: $T_1$ on the left, and $T_2$ on the right}\label{fig:trees1}
\end{figure}

\begin{Thm}\label{thm:McKay1}
Let $T$ be a rooted tree with at least two vertices and with root labelled $0$. Then the trees $\Gamma_i = T\cdot T_i$, $i=1,2$, are not isomorphic, but are co-spectral. Also, their complements $\overline{\Gamma_i}$ and line graphs $L(\Gamma_i)$, $L(\overline{\Gamma_i})$, $\overline{L(\Gamma_i)}$, $i=1,2$, are respectively co-spectral. 
\end{Thm}

Now let $G_1:=G(T_1)$ and $G_2:=G(T_2)$ be the RACGs associated to the trees $T_1$ and $T_2$ defined above.  First note that since the trees $T_1$ and $T_2$ are isomorphic as graphs, the groups $G_1$ and $G_2$ are isomorphic. However, $T_1$ and $T_2$ are not isomorphic as rooted trees. 

Let $\tau$ be a tree with $n$ vertices. Fix a labelling $\{0, \dots, n-1\}$ of the vertices of $\tau$, and suppose that $0$ represents $\tau$'s root. Define $\Gamma_1=\tau \cdot T_1$ and $\Gamma_2=\tau \cdot T_2$ to be the trees obtained as the coalescence of $\tau$ with $T_i$ at vertex $0$, and let $G(\Gamma_1)$ and $G(\Gamma_2)$ be the RACGs based on $\Gamma_1$ and $\Gamma_2$, respectively. Since $\Gamma_1$ and $\Gamma_2$ are non-isomorphic, $G(\Gamma_1)$ and $G(\Gamma_2)$ are non-isomorphic, as well. 
  
The following lemma is the key ingredient of Theorem \ref{thm1}.

\begin{Lem}\label{equal_growth}
The groups $G(\Gamma_1)$ and $G(\Gamma_2)$ have the same geodesic growth series.
\end{Lem}

In order to simplify the exposition in the proof of Lemma~\ref{equal_growth} we use the notation below.
\begin{Not}\label{not:series}

\ \ 
Let $G$ be a group with generating set $T$ containing the letter $0$. Denote the set of words in $Geo(T)$ starting with $0$ by $Geo_0(T)$,  the set of words in $Geo(T)$ ending in $0$ by $Geo^0(T)$, and the set of words in $Geo(T)$ ending and starting with $0$ by $Geo_0^0(T)$.
\end{Not}

\begin{proof} {\it (of Lemma \ref{equal_growth})}
Notice that $f_{Geo(\Gamma_1)}(r)$ (respectively $f_{Geo(\Gamma_2)}(r)$) is equal to the number of all words of length $r$ in $\Gamma_1^*$ (respectively, $\Gamma_2^*$) minus the number of those words of length $r$ in $\Gamma_1^*$ that are not geodesics. We denote the number of non-geodesics by $\overline{f_{Geo(\Gamma_1)}}(r)$ (and $\overline{f_{Geo(\Gamma_2)}}(r)$, respectively). Since $| \Gamma_1| = |\Gamma_2|$, clearly $f_{Geo(\Gamma_1)}(r)=f_{Geo(\Gamma_2)}(r)$ if and only if $\overline{f_{Geo(\Gamma_1)}}(r)=\overline{f_{Geo(\Gamma_2)}}(r)$. 

We now show that $\overline{f_{Geo(\Gamma_1)}}(r)=\overline{f_{Geo(\Gamma_2)}}(r)$ for all $r \geq 1$. Any word in $\Gamma_1^*$ can be written as $w_1u_1w_2 \dots u_n$, where $u_i \in (\tau\setminus\{0\})^*$ and $w_i \in T_1^*$. A non-geodesic word $w$ in $\Gamma_1^*$ belongs to one of the following sets (or is of the type), depending on its form:
\begin{itemize}
\item[$A:$] $w$ contains non-geodesics $u_i \in (\tau\setminus\{0\})^*$ or $w_j \in T_1^* $, or both, for some $1\leq i,j \leq n$, or 

\item[$B:$] all $w_i$ and $u_i$ are geodesic on their respective alphabets, and there exists $1\leq j <n$ such that $0 u_j 0$ is a subword of $w$ with $u_j \in (\Star_{\Gamma_1}(0) \cap \tau)^*$, or

\item[$C:$] all $w_i$ and $u_i$ are geodesic on their respective alphabets, and $w$ contains a subword of the form $s 0 s$, where $s \in (\Star_{\Gamma_1}(0) \cap \tau)^*$.
\end{itemize}

Notice that the set of non-geodesics is then $A \cup B \cup C$, where $A \cap(B \cup C) =\emptyset$ and $B\cap C \neq \emptyset$.

The number of words in $A$ depends only on the geodesic growth series of $\tau$ and $T_1$, and thus it will be equal to the number of words of type $A$ in $\Gamma_2^*$.

The computations in Section \ref{s:equal_special} show that the following identities hold: $f_{Geo(T_1)}(r)=f_{Geo(T_2)}(r)$, $f_{Geo_0(T_1)}(r)=f_{Geo_0(T_1)}(r)$, $f_{Geo^0(T_1)}(r)=f_{Geo^0(T_2)}(r)$ and $f_{Geo_0^0(T_1)}(r)=f_{Geo_0^0(T_2)}(r)$ for all $r \geq 1$, as a result of (\ref{equalities}). This means that there is a length-presenting bijection $\phi_0$ between $Geo_0(T_1)$ and $Geo_0(T_2)$, i.e. for each geodesic $w=0v \in Geo_0(T_1)$ there is a geodesic $w'=0v'=\phi_0(w) \in Geo_0(T_2)$, and $|w|=|w'|$. Analogously, there is a length-preserving bijection $\phi^0:Geo^0(T_1) \rightarrow Geo^0(T_2)$, and a bijection $\phi_0^0:Geo_0^0(T_1) \rightarrow Geo_0^0(T_2)$. This means there is a length-preserving bijection $\phi:Geo_0(T_1)\cup Geo^0(T_1)\rightarrow Geo_0(T_2)\cup Geo^0(T_2)$ between those geodesics starting or ending with $0$, for which we have that $\phi |_{Geo_0} = \phi_0$, $\phi |_{Geo^0} = \phi^0$ and $\phi |_{Geo_0^0} = \phi_0^0$. By the computations in (\ref{equalities}) we have that $f_{(Geo_0 \cup Geo^0)(T_1)}(r)=f_{(Geo_0 \cup Geo^0)(T_2)}(r)$ by the standard formula for the cardinality of the union of two sets, and since $f_{Geo(T_1)}(r)=f_{Geo(T_2)}(r)$, we also have that $f_{Geo \setminus (Geo_0 \cup Geo^0)(T_1)}(r)=f_{Geo \setminus (Geo_0 \cup Geo^0)(T_2)}(r)$. Thus, there is a length-preserving bijection $\psi$ between $Geo \setminus (Geo_0 \cup Geo^0)(T_1)$ and $Geo \setminus (Geo_0 \cup Geo^0)(T_2)$. 

The bijection $\phi$ can be extended to the set $Geo(T_1)$ of all geodesics on $T_1$ by letting $\phi(w)=\psi(w)$ for all $w \in Geo \setminus (Geo_0 \cup Geo^0)(T_1)$, and then furthermore extended to $Geo(T_1) \cup Geo(\tau \setminus \{0\})$ by letting $\phi(w)=w$ for all $w \in Geo(\tau \setminus \{0\}) $.

Then $\phi$ provides a bijection between the words of type $B$ in $\Gamma_1^*$ and the words of type $B$ in $\Gamma_2^*$. To see this, associate to each $w_1u_1w_2 \dots u_n$ the word $\phi(w_1)\phi(u_1) \dots \phi(w_n)\phi(u_n)=\phi(w_1)u_1 \dots \phi(w_n)u_n$. Then $w_i$ ends in $0$, $w_{i+1}$ starts with $0$, and $u_i\in \Star_{\tau}(0)^*$ if and only if $\phi(w_i)$ ends in $0$, $\phi(w_{i+1})$ starts with $0$, and $\phi(u_i)=u_i\in \Star_{\tau}(0)^*$, by definition.

It is immediate to see that $\phi$ also provides a bijection between the words of type $C$, and between the words of type $B \cap C$ in $\Gamma_1$ and $\Gamma_2$, respectively. Thus, there is a length-preserving bijection between the words of type $A \cup B \cup C$ (\textit{i.e.} all non-geodesic words) in $\Gamma_1$ and $\Gamma_2$. This concludes the proof of the lemma.
\end{proof}

\begin{proof} {\it (of Theorem \ref{thm1})}
Let $\Upsilon_k = \{ \tau_1, \tau_2, \dots \}$ be the set of non-isomorphic trees on $k \geq 2$ vertices. For $i=1,2$ the two families of trees $\mathcal{T}^i_n = \{\tau \cdot T_i | \tau \in \Upsilon_k \}$, $k=n-15$, satisfy Theorem \ref{thm1}. 

From Theorem \ref{thm:McKay1} we already know that $\Gamma_1=\tau \cdot T_1$ and $\Gamma_2=\tau \cdot T_2$ are co-spectral, for any $\tau \in \Upsilon_k$. By Lemma \ref{equal_growth} the groups $G(\Gamma_1)$ and $G(\Gamma_2)$ have the same geodesic growth series.
\end{proof}

\begin{Thm}[Lemma 4.3 in \cite{McKay}]\label{thm:McKay2}
Let $p_i(n)$, $i=1,2$, be the proportion of trees on $n$ vertices that have $T_i$ as a limb. Then $p_1(n) = p_2(n)$ for all $n$ and $\lim_{n\rightarrow \infty} p_i(n) = 1$.
\end{Thm}

From Theorem \ref{thm:McKay2}, we obtain that $\frac{\mathrm{card}\, \mathcal{T}^i_n}{\mathrm{card}\, \Upsilon_n} \rightarrow 1$, as $n \rightarrow \infty$.

\section{Computing the numbers of special geodesics}\label{s:equal_special}

In this section we prove, by concrete computations, the equalities between the numbers of special geodesics required in the proof of Lemma \ref{equal_growth}.
 
Let $T_1$ and $T_2$ be as in Figure \ref{fig:trees1}, and recall Notation \ref{not:series}.
For each group $G_i = G(T_i)$, $i=1,2$, we construct a finite automaton $A_i$ accepting the geodesic language of $G_i$, as described in Definition \ref{def:DFA}, and
with the help of $A_i$ we compute the growth series
$$
\gamma_{G_1}(t) = \sum_{r\geq 0}\, f_{Geo(T_1)}(r)\,t^r,\,\, _0\gamma_{G_1}(t) = \sum_{r\geq 0}\, f_{Geo_0(T_1)}(r)\,t^r,
$$
$$
^0\gamma_{G_1}(t) = \sum_{r\geq 0}\, f_{Geo^0(T_1)}(r)\,t^r, \mbox{ and } ^0_0\gamma_{G_1}(t) = \sum_{r\geq 0}\, f_{Geo_0^0(T_1)}(r)\,t^r.
$$
We also compute the analogous growth series $\gamma_{G_2}(t)$, $_0\gamma_{G_2}(t)$, $^0\gamma_{G_2}(t)$ and $^0_0\gamma_{G_2}(t)$ for $G_2$, where the series coefficients are given by the sequences $Geo(T_2)(r)$, $Geo(T_2)_0(r)$, $Geo(T_2)^0(r)$ and $Geo(T_2)^0_0(r)$, respectively. 

Our computations, which we elaborate upon below, show that 
\begin{equation}\label{equalities}
\gamma_{G_1}(t) = \gamma_{G_2}(t), \  _0\gamma_{G_1}(t) = {_0\gamma_{G_2}(t)},  \ ^0\gamma_{G_1}(t)= {^0\gamma_{G_2}(t)} \ \textrm{and} \ ^0_0\gamma_{G_1}(t) = {^0_0\gamma_{G_2}(t)}. 
\end{equation}

These identities prove the equality of the corresponding numbers of geodesics. 

Let $A$ be a deterministic finite-state automaton with accepting states $q_i$, $i=0,1,\dots, N$, where $q_0$ is the start state and the ``fail'' state is denoted by $q$. Let $M = M(A)$ be the transition matrix of the automaton $A$. Computing the generating function $\gamma_A (t)$ of $A$ is a standard technique, and the formula for $\gamma_A(t)$ is
\begin{equation} \label{DFAseries}
\gamma_A(t) = \frac{e^T\, \widetilde{M} w}{\det(I - t M)},
\end{equation}
where $I$ is the $N \times N$ identity matrix, $\widetilde{M}$ is the adjoint matrix of $I - t M$, $e=(1, 0, \dots, 0)^T$ and $w=(1,1, \dots, 1)^T$ are two vectors in $\mathbb{Z}^N$.
 
Now let $A_i$ be the deterministic finite-state automata accepting the language of geodesics in $G_i$, $i=1,2$, and 
 let $M_i = M(A_i)$ be the transition matrix of $A_i$. Then the geodesic growth series $\gamma_{G_i}(t)$ is a rational function (determined by the equality (\ref{DFAseries})), and the coefficients of its numerator and denominator can be easily computed, see \cite{Epstein}. 
The Python code $\mathrm{Monty}$ \cite{Pyth} may perform the above computation either by finding the symbolic determinant $\det(I - t M_i)$ (usually slow), or by applying the Berlekamp-Massey algorithm (a faster one). This Python code starts by creating the finite-state automaton $A_i$, given a triangle-free graph $T_i$ (in this case, a tree), and then proceeds to determining $\gamma_{G_i}(t)$. 

Since $G_1$ and $G_2$ are isomorphic, the equality $\gamma_{G_1}(t) = \gamma_{G_2}(t)$ is immediate. Here we provide an explicit formula for these identical growth series. The output of $\mathrm{Monty}$ for both trees $T_1$ and $T_2$ consists of two finite-state automata $A_i$ which are isomorphic, as one only renumbers the vertices of $T_2$ in order to obtain $T_1$. Each $A_i$ has $32$ states and $466$ transition arrows, and the corresponding growth series are
\begin{flalign*}
&\gamma_{G_1}(t) = \gamma_{G_2}(t) = (1+t)(1 + 2t - 2t^3 - 4t^4 - t^5)(1 + 5t + 10t^2 + 9t^3 - 5t^4 - 26t^5 - 34t^6&\\
& - 22t^7 - t^8 + 7t^9 + 4t^{10}) (1 - 8t - 85t^2 - 243t^3 - 222t^4 + 332t^5 + 1194t^6 + 1349t^7 + 132t^8&\\
& - 1510t^9 - 2008t^{10} - 1088t^{11} + 28t^{12} + 359t^{13} + 170t^{14} + 15t^{15})^{-1}.& 
\end{flalign*}

Now we compute the functions $_0\gamma_{G_i}(t)$, $i=1,2$, which will turn out to be equal, as well. However, in this case the equality is not known to hold beforehand: even though the groups $G_1$ and $G_2$ are isomorphic, the image of vertex $0$ in $T_1$ under the canonical isomorphism does not correspond to vertex $0$ in $T_2$. 
We shall explicitly compute the generating function $_0\gamma_{G_i}(t)$ for the number of geodesic words starting with $0$ accepted by each $A_i$, $i=1,2$, and then compare these functions. Suppose that the start state of each $A_i$ is $q_0 = \emptyset$. Let the corresponding transition function $\delta_i$ be so that $\delta_i(q_0, 0) = q_{k_i} = \{0\}$. If a word $w$ labels a path from $q_{k_i}$ to an accept state, then the word $0w$ is a geodesic word in $G_i$ starting with $0$. 
Thus, we have to compute the generating function $_0\alpha_{G_i}(t)$ for the number of words starting at $q_{k_i}$ and ending at an accept state:
\begin{equation*}
_0\alpha_{G_i}(t) = \frac{e^T\, \widetilde{M_i}\, w}{\det(I - t M_i)},
\end{equation*}
where $e = (0, \dots, 0, \underbrace{1}_{k_i}, 0, \dots, 0)^T$ and $w = (1, 1, \dots, 1)^T$.
Then we use the fact that $_0\gamma_{G_i}(t) = t \cdot {_0\alpha_{G_i}(t)}$. By symmetry, we get $^0\gamma_{G_i}(t) = {_0\gamma_{G_i}(t)}$, $i=1,2$.

By using $\mathrm{Monty}$ we obtain 
\begin{flalign*}
&_0\gamma_{G_1}(t) = {_0\gamma_{G_2}(t)} = t (1+t) (1 + 2t - 2t^3 - 4t^4 - t^5) (1 + 4t + 4t^2 - 3t^3 - 9t^4 - 5t^5 + 3t^6&\\ 
& + t^7 - 3t^8 - 3t^9) (1 - 8t - 85t^2 - 243t^3 - 222t^4 + 332t^5 + 1194t^6 + 1349t^7 + 132t^8&\\
& - 1510t^9 - 2008t^{10} - 1088t^{11} + 28t^{12} + 359t^{13} + 170t^{14} + 15t^{15})^{-1}.&
\end{flalign*} 

Finally, it remains to compute the growth series $^0_0\gamma_{G_i}(t)$. 

There are two kinds of words forming disjoint subsets of the geodesic language of $G_i$ that we are interested in:
\begin{itemize}
\item[(I)] the words $w$ starting at the accept state $q_{k_i} = \delta(q_0, 0) = \{0\}$ of $A_i$ and coming back to it: then the geodesic word $0w$ starts and ends with a ``0'' (since there are only $0$-transitions leading to $q_{k_i} = \{0\}$ and no $0$-transition coming out of $q_{k_i}$), by Definition \ref{def:DFA} (i).

\item[(II)] the words $w$ starting at the state $q_{k_i} = \{0\}$ and ending at a state $q_{l_i}$ such that $\delta(q_{l_i}, 0) = q_{m_i} \neq q_{k_i}$ and $q_{m_i}$ is an accept state: then the word $0w0$ will be a geodesic word starting and ending with a ``0''. 
\end{itemize}

Let $^0_0\alpha_{G_i}(t)$ be the generating function for the words of type (I), and $^0_0\beta_{G_i}(t)$ be that for the words of type (II). Then, $^0_0\gamma_{G_i}(t) = t\cdot {^0_0\alpha_{G_i}(t)} + t^2 \cdot {^0_0\beta_{G_i}(t)}$. 
We have that 
\begin{equation*}
^0_0\alpha_{G_i}(t) = \frac{e^T\, \widetilde{M_i}\, w}{\det(I - t M_i)},
\end{equation*}
with $e = w = (0, \dots, 0, \underbrace{1}_{k_i}, 0, \dots, 0)^T$.

Analogously,
\begin{equation*}
^0_0\beta_{G_i}(t) = \frac{e^T\, \widetilde{M_i}\, w}{\det(I - t M_i)},
\end{equation*}
with $e = (0, \dots, 0, \underbrace{1}_{k_i}, 0, \dots, 0)^T$ and $w$ having a ``1'' at position $l_i$ for all states $q_{l_i}$ as described above, and zeroes at all other places.

By using $\mathrm{Monty}$ we obtain
\begin{flalign*}
&^0_0\gamma_{G_1}(t) = {^0_0\gamma_{G_2}(t)} = t\cdot (1 - 5t - 94t^2 - 374t^3 - 456t^4 + 955t^5 + 4275t^6 + 5652t^7 - 1617t^8&\\
& - 16773t^9 - 24255t^{10} - 7337t^{11} + 26583t^{12} + 45100t^{13} + 26181t^{14} - 12789t^{15} - 34553t^{16}&\\
& - 24957t^{17} - 3147t^{18} + 8130t^{19} + 6288t^{20} + 1398t^{21} - 458t^{22} - 284t^{23} - 24t^{24})\cdot(1 + 3t&\\
& + 2t^2 - 3t^3 - 9t^4 - 8t^5 + 4t^7 + 3t^8 - t^9)^{-1} (1 - 8t - 85t^2 - 243t^3 - 222t^4 + 332t^5 + 1194t^6&\\
& + 1349t^7 + 132t^8 - 1510t^9 - 2008t^{10} - 1088t^{11} + 28t^{12} + 359t^{13} + 170t^{14} + 15t^{15})^{-1}.&
\end{flalign*}

Thus, the numbers of special geodesics in $G_1$ and $G_2$ coincide.  

\section{RACGs with different geodesic growth}\label{s:racg-diff-growth}

In this section we give a proof of Theorem~\ref{thm2}. Our construction will be analogous to that in Section~\ref{s:racg}, although we shall use different trees, $S_1$ and $S_2$, in order to construct the families $\mathcal{S}^i_n = \{ \tau \cdot S_i | \tau \in \Upsilon_k \}$, $k = n-9$. Namely, we will use the rooted trees in Fig.~\ref{fig:trees2}. 

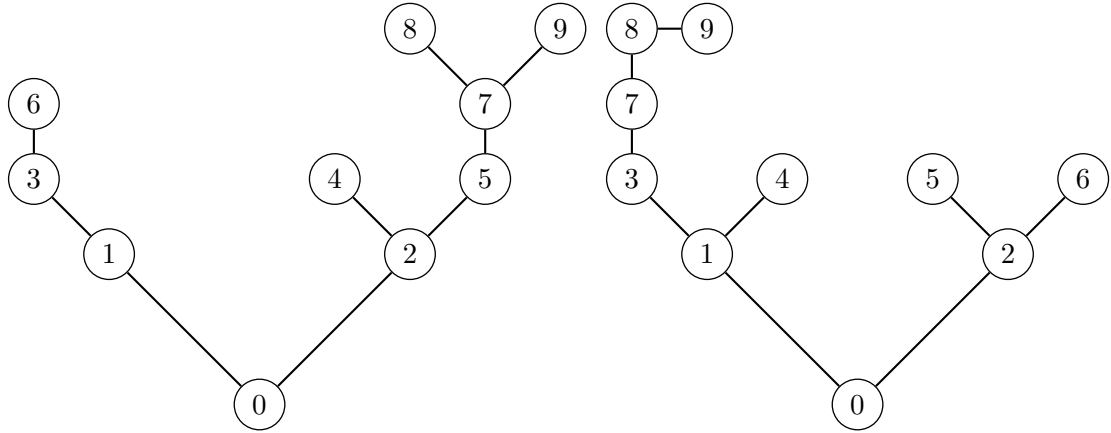
\begin{figure}[ht]
\centering
\begin{subfloat}{}
\begin{tikzpicture}
\GraphInit[vstyle=Normal]
\SetGraphUnit{2}
\Vertex{0}
\NOWE(0){1} \NOEA(0){2} 
\Edge(0)(1) \Edge(0)(2)
\SetGraphUnit{1}
\NOWE(1){3} \NOWE(2){4} \NOEA(2){5}
\Edge(1)(3) \Edge(2)(4) \Edge(2)(5)
\NO(3){6} \NO(5){7} \NOWE(7){8} \NOEA(7){9}
\Edge(3)(6) \Edge(5)(7) \Edge(7)(8) \Edge(7)(9)
\end{tikzpicture}
\end{subfloat}
~
\begin{subfloat}{}
\begin{tikzpicture}
\GraphInit[vstyle=Normal]
\SetGraphUnit{2}
\Vertex{0}
\NOWE(0){1} \NOEA(0){2} 
\Edge(0)(1) \Edge(0)(2)
\SetGraphUnit{1}
\NOWE(1){3} \NOEA(1){4} \NOWE(2){5} \NOEA(2){6}
\Edge(1)(3) \Edge(1)(4) \Edge(2)(5) \Edge(2)(6)
\NO(3){7} \NO(7){8} \EA(8){9}
\Edge(3)(7) \Edge(7)(8) \Edge(8)(9)
\end{tikzpicture}
\end{subfloat}
\caption{Godsil's rooted trees: $S_1$ on the left, and $S_2$ on the right}\label{fig:trees2}
\end{figure}

We set $\Gamma_1 = \tau \cdot S_1$ and $\Gamma_2 = \tau \cdot S_2$, where $\tau$ is an arbitrary tree with $n$ vertices labelled $\{0,\dots,n-1\}$. Suppose that $0$ represents $\tau$'s root. Then the following holds:

\begin{Thm}[page 27 of \cite{Godsil-slides}]\label{thm:Godsil}
The trees $\Gamma_1$ and $\Gamma_2$ are not isomorphic, but are co-spectral and have co-spectral complements. Their line graphs $L(\Gamma_i)$, $L(\overline{\Gamma_i})$ and $\overline{L(\Gamma_i)}$ are not necessarily co-spectral. 
\end{Thm} 

\begin{proof} The trees $\Gamma_1$ and $\Gamma_2$ are not isomorphic, since $\Gamma_1$ has the rooted tree $\sigma$ depicted in Fig.~\ref{fig:trees3} as a limb more times than $\Gamma_2$ does.

\begin{figure}[ht]
\centering
\begin{tikzpicture}
\GraphInit[vstyle=Normal]
\SetGraphUnit{2}
\Vertex{0}
\NOEA(0){2} 
\Edge(0)(2)
\SetGraphUnit{1}
\NOWE(2){4} \NOEA(2){5}
\Edge(2)(4) \Edge(2)(5)
\NO(5){7} \NOWE(7){8} \NOEA(7){9}
\Edge(5)(7) \Edge(7)(8) \Edge(7)(9)
\end{tikzpicture}
\caption{The rooted tree $\sigma$ with root $0$}\label{fig:trees3}
\end{figure}
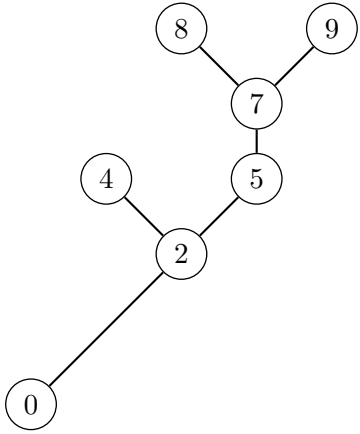

The characteristic polynomials of $\tau$, $S_i$ and $\Gamma_i$ ({\it i.e.} the characteristic polynomials of the adjacency matrices of these graphs) are related by
\begin{equation*}
\phi_{\Gamma_i}(t) = \phi_{\tau}(t) \cdot \phi_{S_i\setminus \{0\}}(t) + \phi_{\tau\setminus \{0\}}(t) \cdot \phi_{S_i}(t) - t\cdot \phi_{\tau\setminus \{0\}}(t) \cdot \phi_{S_i\setminus \{0\}}(t),
\end{equation*}
according to \cite[Lemma 2.2 (i)]{McKay}. Given that the trees $S_1$ and $S_2$ are isomorphic as graphs (though not as rooted trees), we obtain that $\Gamma_1$ and $\Gamma_2$ share the same characteristic polynomial ({\it i.e.} are co-spectral). 

The fact that $\overline{\Gamma_1}$ and $\overline{\Gamma_2}$ are co-spectral follows from the above and \cite[Theorem 3.1 (i)]{McKay}.

Now, by letting $\tau$ be the line graph on two vertices we obtain that $L(\Gamma_1)$, $L(\overline{\Gamma_1})$ and $\overline{L(\Gamma_1)}$ have different characteristic polynomials from $L(\Gamma_2)$, $L(\overline{\Gamma_2})$ and $\overline{L(\Gamma_2)}$, respectively.
\end{proof}

\begin{Lem}\label{lem3}
Let $\tau$ be a rooted tree with at least two vertices, and let $\Gamma_1 = \tau \cdot S_1$ and $\Gamma_2 = \tau \cdot S_2$. The RACGs $G(\Gamma_1)$ and $G(\Gamma_2)$ have distinct geodesic growth series. 
\end{Lem}

\begin{proof} We will prove that $f_{Geo(\Gamma_1)}(10) \neq f_{Geo(\Gamma_2)}(10)$. We use the same argument as in the proof of Lemma \ref{equal_growth}, that is, $f_{Geo(\Gamma_1)}(r)=f_{Geo(\Gamma_2)}(r)$ if and only if $\overline{f_{Geo(\Gamma_1)}}(r)=\overline{f_{Geo(\Gamma_2)}}(r)$, where $\overline{f_{Geo(\Gamma_i)}}(r)$ denotes the numbers of non-geodesics of length $r$, $r \geq 1$. In this proof we use the result of our explicit computations with Monty, which shows that $f_{Geo_0(\Gamma_1)}(r)= f_{Geo_0(\Gamma_2)}(r)$ for $r<8$, but $f_{Geo_0(\Gamma_1)}(8) = f_{Geo^0(\Gamma_1)}(8) = 8919523$ for $G_1$ and $f_{Geo_0(\Gamma_2)}(8) = f_{Geo^0(\Gamma_2)}(8) = 8919522$ for $G_2$.

Any word in $\Gamma_1^*$ has the form $w_1u_1w_2 \dots u_n$, where $u_i \in (\tau\setminus \{0\})^*$, $w_i \in S_1^*$, and $w_i, u_i$ non-empty except for perhaps $w_1$ and $u_n$. A non-geodesic word $w$ in $\Gamma_1^*$ belongs to one of the following sets (or is of the type), depending on its form:
\begin{itemize}
\item[(A)] it either contains non-geodesics $w_i \in (\tau \setminus \{0\})^*$ or $u_j \in S_1^* $, or both, for some $1\leq i,j \leq n$, or 
\item[(B)] all $w_i$ and $u_i$ are geodesic on their respective alphabets, and there exists $1\leq j <n$ such that $0 u_j 0$ is a subword of $w$ and $u_j \in (\Star_{\Gamma_1}(0) \cap \tau)^*$, or
\item[(C)] all $w_i$ and $u_i$ are geodesic on their respective alphabets, and $w$ contains a subword of the form $s 0 s$, where $s \in (\Star_{\Gamma_1}(0) \cap \tau)^*$.
\end{itemize}

For the remaining discussion we only consider words $w$ of length $10$, and call the number of non-empty subwords $u_i$ or $w_i$ the \textit{syllable length} of $w$.  The number of words of type (A) depends only on the geodesic growth series of $\tau$ and $S_1$, and thus it will be equal to the number of words of type (A) in $\Gamma_2$.
Notice that if the syllable length of $w$ is $\leq 2$ there are no words of type (B) or (C), so we obtain the same numbers of words in $\Gamma_1$ and $\Gamma_2$. If the syllable length of $w$ is $>4$ then all $w_i, u_i$ are shorter than or equal to $6$, and our computations show that the numbers of special geodesics of length less than or equal to $6$ in $\Gamma_1$ and $\Gamma_2$ coincide. Thus, a discrepancy may appear only when the syllable length is $3$ or $4$. The words of type (B) or (C) with syllable length $4$ have the form $w= w_1 u_1 w_2 u_2$ (or $w=  u_1 w_2 u_2 w_3$), where $w_1$ ends in $0$ and $w_2$ starts with $0$, or $w_2=0$. But in this case $|w_1|, |w_2| \leq 7$, and our computations show that $f_{Geo_0(\Gamma_1)}(r)=f_{Geo_0(\Gamma_2)}(r)$, up to $r=7$.

Next, we consider the non-geodesics $w$ of syllable length $3$. If they have the form $w=u_1 w_1 u_2$, then $w_1=0$, and the numbers of geodesics $u_i$ is the same since they are all written over the same alphabet determined by the tree $\tau$, so no discrepancy in the numbers of non-geodesics occurs. Thus, it remains to count the words of the form $w=w_1 u_1 w_2$, where $|w_1| \leq 8$, $w_1$ ends in $0$, $|u_1| \geq 1$, $u_1 \in (\Star_{\Gamma_1}(0) \cap \tau)^*$, and $w_2$ starts with $0$, $|w_2| \leq 8$. Again, if $|w_i|\leq 7$ we obtain the same numbers of non-geodesics. However, a discrepancy occurs when $|w_1|=8$ or $|w_2|=8$. 

Indeed, the number of such words is $2(f_{Geo_0(\Gamma_1)}(8)\deg_{\tau}(0))$ in $\Gamma_1^*$ and $2(f_{Geo_0(\Gamma_2)}(8)\deg_{\tau}(0))$ in $\Gamma_2^*$. By using Monty, we obtained $f_{Geo_0(\Gamma_1)}(8) = f_{Geo^0(\Gamma_1)}(8) = 8919523$ for $G_1$, but $f_{Geo_0(\Gamma_2)}(8) = f_{Geo^0(\Gamma_2)}(8) = 8919522$ for $G_2$, so the numbers of non-geodesics in these two groups are distinct.
\end{proof}

Now we can finish the proof of our second main result.

\begin{proof} {\it (of Theorem~\ref{thm2})} This is a straightforward consequence of Theorem~\ref{thm:Godsil} and Lemma~\ref{lem3}. \end{proof}

\medskip

\section*{Acknowledgments}

The authors gratefully acknowledge the support received from the Swiss National Science Foundation: L.C. was supported by PP00P2-144681 (SNSF Professorship); A.K. was supported by P300P2-151316 (Advanced Post-Doc Mobility) and P300P2-151316/2 (CH-Link). The authors are also thankful to Prof. Ruth Kellerhals (University of Fribourg, Switzerland), Prof. Michelle Bucher (University of Geneva, Switzerland), Prof. Robert Young (Courant Institute of Mathematical Sciences, New York, USA) and Dr. Alexey Talambutsa (Steklov Mathematical Institute, Moscow, Russia) for fruitful discussions. 

\medskip

\section*{Appendix}

In this section we give the Python code ``Monty'' that we used in our computations, with comments and remarks. Its on-line copy \cite{Pyth} can be downloaded as a SAGE worksheet from the second author's \href{www.sashakolpakov.wordpress.com/list-of-papers}{web-page.}

We begin by defining the automaton $A$ that recognizes the language of geodesics of a given RACG $G$ whose defining graph $\Gamma$ is given as input.

\begin{verbatim}
def Automaton(t): 
# takes a triangle-free graph, returns the automaton recognising 
# the resp. RACG as a digraph
    CliqueComplex = t.clique_complex(); C = list();
    for s in CliqueComplex.faces().values(): 
        for f in s: C.append(set(f));
    n = len(C); a = DiGraph(); a.add_vertices(range(n)); 
    for i in range(n):
        for v in t.vertices():
            if not(C[i].issuperset([v])):
                st = set(t.vertex_boundary([v])).union([v]);            
                d = set([v]).union(st.intersection(C[i]));
                k = C.index(d);
                a.add_edge((i,k));
    a.set_vertices({i : C[i] for i in range(n)})
    return a;
\end{verbatim}

Given the automaton $A$, we then compute the growth function of its accepted language (in this case, the geodesic language for $G$ with respect to $S$).

\begin{verbatim}
def GrowthFunc(a): 
# takes a geodesic automaton, returns the resp. (geodesic) growth function
    am = a.adjacency_matrix();
    n  = am.nrows();
    R = FractionField(PolynomialRing(QQ, 't')); 
    R.inject_variables(); 
    m = diagonal_matrix([1]*n) - t*am; 
    M = (1/m.det())*m.adjoint();
    ee = [0]*(n-1); 
    ee.append(1); 
    e = vector(ee); 
    w = vector([1]*n);
    func = e*M*w;
    return func.numerator().factor()/func.denominator().factor();
\end{verbatim}

A clique $c$ in the defining graph $\Gamma$ of the RACG $G = G(\Gamma)$ corresponds to a state $q_c$ in the automaton $A$. We need the following auxiliary function in order to determine the index of $q_c$ represented as a vertex of the digraph $A$ (the automaton) created by the procedure \texttt{Automaton}. 

\begin{verbatim}
def Index(a, c): 
# takes a geodesic automaton `a', a clique `c' in the resp. defining graph, 
# returns the vertex of the automaton `a' corresponding to `c'
    l = None;
    for v in a.vertices():
        if a.get_vertex(v) == c:
            l = v;
    return l;
\end{verbatim}

Below we compute the growth function ${_0}\alpha(t)_G$ as described in the proof of Theorem~\ref{thm1}. 

\begin{verbatim}
def GrowthFuncStart0(a): 
# takes a geodesic automaton, returns the growth function for geodesic words 
# starting at the state q, where $\delta(Start, 0) = q$
    am = a.adjacency_matrix();
    n  = am.nrows();
    R = FractionField(PolynomialRing(QQ, 't')); 
    R.inject_variables(); 
    m = diagonal_matrix([1]*n) - t*am; 
    M = (1/m.det())*m.adjoint();
    ind = Index(a, set([0]));
    ee = [0]*n; 
    ee[ind] =  1; 
    e = vector(ee); 
    w = vector([1]*n);
    func = e*M*w;
    return func.numerator().factor()/func.denominator().factor();
\end{verbatim} 

Now we define a function that takes as input the geodesic automaton $A$ for a RACG $G = G(\Gamma)$, a list of cliques $l = [c_0, c_1, \dots, c_k]$ in the respective defining graph $\Gamma$ and returns the growth function for geodesic words starting at the state $q = \delta(q_0, 0)$ that bring $A$ to any of the states described by the cliques in $l$. 

\begin{verbatim}
def GrowthFuncStart0End(a, l): 
# takes an automaton and a list of cliques `l' in the resp. defining graph
# as input,  returns the growth function for geodesic words starting at 
# state q, where $\delta(Start, 0) = q$, and ending at any of the states 
# corresponding to cliques in `l'
    am = a.adjacency_matrix();
    n  = am.nrows();
    R = FractionField(PolynomialRing(QQ, 't')); 
    R.inject_variables(); 
    m = diagonal_matrix([1]*n) - t*am; 
    M = (1/m.det())*m.adjoint();
    ee = [0]*n;
    ind = Index(a, set([0])); 
    ee[ind] =  1; 
    e = vector(ee); 
    ww = [0]*n;
    for c in l:
        ind = Index(a, set(c));
        ww[ind] = 1;
    w = vector(ww);
    func = e*M*w;
    return func.numerator().factor()/func.denominator().factor();
\end{verbatim}

By using a suitable list of cliques $l$ we can compute the functions ${_0^0}\alpha(t)_G$ and ${_0^0\beta(t)}_G$. Namely, in the proof of Theorem~\ref{thm1}, we find

\begin{verbatim}
# the function ${_0^0}\alpha(t)_{G_1}$
a001 = GrowthFuncStart0End(a1, [[0]]); 
# the function ${_0^0}\alpha(t)_{G_2}$
a002 = GrowthFuncStart0End(a2, [[0]]);
# the function ${_0^0}\beta(t)_{G_1}$
b001 = GrowthFuncStart0End(a1, [[1], [1,3], [1,4], [2], [2,5]]);
# the function ${_0^0}\beta(t)_{G_2}$
b002 = GrowthFuncStart0End(a2, [[1], [1,3], [2], [2,4], [2,5]]);
\end{verbatim}

The list \texttt{l = [[1], [1,3], [1,4], [2], [2,5]]} above contains the cliques of $\Gamma_1$ corresponding to the accept states $q$ of the geodesic automaton $A_1$ for $G_1 = G(\Gamma_1)$ such that $\delta(p, 0) = q$, for a state $p$. The list \texttt{l = [[1], [1,3], [2], [2,4], [2,5]]} contains the cliques of $\Gamma_2$ with analogous properties, corresponding to the states of the geodesic automaton $A_2$ for the RACG $G_2 = G(\Gamma_2)$.

The above described Python procedures are also used to perform the computations in the proof of Theorem~\ref{thm2}. 

The on-line version of Monty \cite{Pyth} contains a variation of the \texttt{GrowthFunc} procedure, called \texttt{GrowthFuncBM}, that uses the Berlekamp-Massey algorithm for faster computing.

\bigskip

\begin{flushleft}
\textit{Laura Ciobanu\\
Department of Mathematics\\
University of Neuch\^atel\\
Rue Emile - Argand 11\\
CH-2000 Neuch\^atel, Switzerland\\
}

\emph{e-mail}{:\;\;}\texttt{laura.ciobanu@unine.ch}\\

\vspace*{0.25in}

\textit{Alexander Kolpakov\\
Department of Mathematics\\
University of Toronto\\
40 St. George Street\\
M5S 2E4 Toronto ON, Canada\\
}

\emph{e-mail}{:\;\;}\texttt{kolpakov.alexander@gmail.com}\\
\end{flushleft}

\newpage
\hoffset 0cm \voffset 0cm \textwidth=6.2in \textheight=8.3in
\tolerance=9000 \emergencystretch=5pt \vfuzz=2pt
\parskip=1.2mm

\end{document}